\documentclass[11pt]{amsart}
\usepackage{amsmath, amssymb, amsthm, color}
\usepackage{longtable}
\usepackage{amsfonts}
\usepackage{setspace} 
\usepackage{hyperref} 
\usepackage{graphicx}
\usepackage{latexsym}
\usepackage{array}
\numberwithin{equation}{section}

\theoremstyle{plain}

\newtheorem{thm}{Theorem}[section]
\newtheorem{prop}[thm]{Proposition}
\newtheorem{lemm}[thm]{Lemma}

\newtheorem*{acknowledgements}{Acknowledgements}

\theoremstyle{remark}
\newtheorem{rema}[thm]{Remark}

\newcommand{\bq}{\begin{equation}}
\newcommand{\eq}{\end{equation}}
\providecommand{\abs}[1]{\left\lvert#1\right\rvert}
\providecommand{\norm}[1]{\left\lVert#1\right\rVert}
\providecommand{\ip}[1]{\langle#1\rangle}
\providecommand{\iFT}[1]{\mathcal{F}^{-1}\left(#1\right)}

\renewcommand{\div}{{\rm div}}

\newcommand{\Z}{\mathbb{Z}}

\newcommand{\R}{\mathbb{R}}

\address{Department of Mathematics, Princeton University, Princeton, NJ 08544, USA}
\email{tme2@math.princeton.edu}

\address{Courant Institute of Mathematical Sciences, 251 Mercer Street, New York
10012 NY, USA}
\email{klaus@cims.nyu.edu}

\begin{document}

\title[The Dispersive SQG equation and the inviscid Boussinesq system]{Sharp decay estimates for an anisotropic linear semigroup and applications to the SQG and inviscid Boussinesq systems}
\author{Tarek M. Elgindi and Klaus Widmayer}
\subjclass[2000]{76B03, 76B15, 35B35, 35Q31, 35Q35}
\keywords{dispersive surface quasi-geostrophic equation, inviscid Boussinesq system, stability}

\begin{abstract}
At the core of this article is an improved, sharp dispersive estimate for the anisotropic linear semigroup $e^{R_1t}$ arising in both the study of the dispersive SQG equation and the inviscid Boussinesq system. We combine the decay estimate with a blow-up criterion to show how dispersion leads to long-time existence of solutions to the dispersive SQG equation, improving the results obtained using hyperbolic methods. In the setting of the inviscid Boussinesq system it turns out that linearization around a specific stationary solution leads to the same linear semigroup, so that we can make use of analogous techniques to obtain stability of the stationary solution for an increased timespan.
\end{abstract}

\maketitle

\tableofcontents

\section{Introduction}

The key ingredient of this article is an improved, sharp dispersive estimate for the semigroup $e^{R_1 t}$ of the linear equation
\begin{equation*}
 \partial_t\theta-R_1\theta=0,\qquad \theta:\R^+\!\!\times\R^2\rightarrow\R.
\end{equation*}
Here, $R_1$ is the Riesz transform with Fourier multiplier with symbol $-i\frac{\xi_1}{\abs{\xi}}.$ The presence of the Riesz transform makes this an an-isotropic equation and introduces a degeneracy from the point of view of stationary phase methods. 

Equipped with such a dispersive estimate we can perturbatively treat nonlinear versions of the above equation. More specifically, here we focus on two problems of independent interest, both with quadratic nonlinearities: The dispersive surface quasi-geostrophic equation (SQG) and the inviscid Boussinesq system, discussed in the following.

\subsection{The SQG Equation}\label{ssec:sqg_intro}
We are concerned with the time of existence of solutions $\theta:\R^+\!\!\times\R^2\rightarrow\R$ to the dispersive surface quasi-geostrophic equation (SQG)
\begin{equation*}
 \partial_t\theta+u\cdot\nabla\theta=R_1\theta,
\end{equation*}
where $u=\nabla^\perp(-\Delta)^{-1/2}\theta$. This equation has been suggested as a model in geophysical fluid dynamics and may prove useful in the study of wave-turbulence interactions (see \cite{AIP_1.3141499}). As the name suggests, this fluid equation exhibits dispersion.

The inviscid and viscous SQG equations have been studied by many authors, particularly because they are expected to model the main features of $3d$ incompressible flow (see \cite{MR1304437}).  We mention only a few results from the literature: Local existence and uniqueness for $H^{2+}$ solutions and global existence of $L^2$ solutions were established in \cite{Resnick}. Blow-up criteria of Beale-Kato-Majda type were obtained in \cite{MR1304437}. In recent years, the critically dissipative SQG equation was proven to be globally well-posed in $H^s$ spaces for $s\geq 1$ by several authors using a variety of methods (\cite{MR2276260}, \cite{MR2680400}, \cite{MR2989434}, \cite{MR2586369}). Aside from local well-posedness and breakdown criteria, not much is known about the well-posedness of the inviscid SQG equation. 

On the other hand, the addition of the Riesz transform linearity $R_1\theta$ on the right hand side of the inviscid SQG equation $\partial_t\theta+u\cdot\nabla\theta=0$ allows for this problem to be studied from a perturbative point of view. This is a standard approach and has been very fruitful in the setting of dispersive partial differential equations, especially when the decay properties of the linear equation are very strong. Notice, however, that in the case of the dispersive SQG equation the Riesz transform breaks the isotropy of the equation. Apart from other difficulties this entails  that the optimal decay rate of the linear semigroup associated to the problem is slower than that of typical $2d$ problems.

The available existence theory for this equation is very basic. From the hyperbolic systems point of view the dispersive and the inviscid SQG equations have the same energy estimates, thus solutions (for either) exist a prior only for a timespan inversely proportional to the size of the initial data. A finer understanding of the dispersive properties allows us to improve this timespan. In particular, here it is crucial to obtain a sharp decay result, which has not been done before for such equations.

In comparison, other works on this type of equations have studied them in the presence of a viscous or dissipative term. This typically resolves the existence question (for small data or a large dispersion constant) and in some cases allows for the use of the dispersion to derive equations in limiting cases (see for example the case of the $\beta$-plane equation in \cite[Chapter 5]{MR2228849}, or the work of Babin et al.\ \cite{MR1752139} on the Navier-Stokes equations with Coriolis forcing). In this context the exact dispersive rate is not of key importance.

In a related work \cite{MR3048553}, Cannone, Miao and Xue proved the existence of global strong solutions to the dispersive SQG equation with supercritical dissipation. Specifically, the authors of  \cite{MR3048553} study
$$\partial_t \theta+u\cdot\nabla\theta=-\nu(-\Delta)^{\alpha}\theta+KR_{1}\theta,\qquad 0<\alpha<1,$$
and prove that for given initial data $\theta_0$ there exists $K$ large enough such that the solution is global. They use the dissipation in a crucial way, whereas our model does not have any (i.e.\ we study the case where $\nu=0$). On the other hand, the rate of dispersive decay is not central to this argument and our improvement of the dispersive estimate allows for lowering the constant $K$ for which one obtains a global solution.

On a separate note we remark that the dispersive SQG equation might be viewed as a natural generalization of the (not dispersive) Burgers-Hilbert Equation studied by Ifrim and Tataru in \cite{2013arXiv1301.1947H}. Their use of a new, modified energy method is based on the existence of a normal form. However, for the dispersive SQG equation such a transformation does not seem to be available.

\subsection{The Inviscid Boussinesq System}\label{ssec:bouss_intro}
We consider an incompressible, inviscid fluid $v:\R^+\!\!\times\R^3\rightarrow\R^3$ of variable density $\rho:\R^+\!\!\times\R^3\rightarrow\R$ under the influence of an external gravity force proportional to $\rho$ in direction $\mathbf{k}\in\R^3$, described by Euler's equation coupled to a continuity equation and an equation of state:
\begin{equation*}
\begin{cases}
 &\rho\left(\partial_t v+v\cdot\nabla v\right)=-\nabla p+g \rho \mathbf{k},\\
 &\partial_t\rho + v\cdot\nabla\rho=0,\\
 &\div \,v=0,
\end{cases}
\end{equation*}
where $p:\R^+\!\!\times\R^3\rightarrow\R$ is the fluid pressure and $g$ is a gravitational constant.

Here we will further investigate the \emph{Boussinesq approximation} of these equations, where the density variation is assumed to be small in comparison with the effects of gravity. For $2d$ flows -- choosing $\mathbf{k}=e_2$ and $g=1$ -- these equations can then be reduced to a system of scalar vorticity equations (see \cite{MR1925398}, \cite{2014arXiv1402.6499H}), the so-called \emph{inviscid Boussinesq system}:
\begin{equation}\label{eq:bouss_intro}
\begin{cases}
 &\partial_t \omega+u\cdot\nabla\omega=\partial_x\rho,\\
 &\partial_{t}\rho+u\cdot\nabla\rho=0,\\
 &u=\nabla^\perp(-\Delta)^{-1}\omega,
\end{cases}
\end{equation}
where $\omega,\rho:\R^+\!\!\times\R^2\rightarrow\R$.

The question of global well-posedness for this equation remains unresolved. From a hyperbolic systems point of view these equations bear a lot of similarity to the SQG equation. In particular, using energy methods and Beale-Kato-Majda type blow-up criteria, local existence results in various spaces have been obtained by several authors (\cite{MR1475638}, \cite{MR1711383}, \cite{MR2645152}, \cite{MR2926890}).

In the present article we are concerned with the stability of certain stationary solutions of the Boussinesq system \eqref{eq:bouss_intro} -- in one sentence, we are going to study the opposite of the Rayleigh-Bernard instability, namely the setup where warm fluid is on top of cooler fluid. We will prove long-time stability \emph{in the absence of viscosity}. Here the temperature and density are assumed to be proportionally related, so that the cooler fluid is more dense. The gravitational force is thus expected to stabilize such a density (or temperature) distribution.

The phenomenon known as Rayleigh-Bernard convection has been studied by a number of authors for many years. The idea is simple: take a container filled with water which is at rest. Now heat the bottom of the container and cool the top of the container. It has been observed experimentally and mathematically that if the temperature difference between the top and the bottom goes beyond a certain critical value, the water will begin to move and convective rolls will begin to form. This effect is called Rayleigh-Bernard instability (\cite{DrazinReid},\cite{DoeringGibbon}).  

Now, in the other case, when one cools the bottom and heats the top, it is expected that the system remains stable. In the presence of viscosity it is not difficult to prove this fact (see \cite{DoeringGibbon}). However, without the effects of viscosity (or temperature dissipation), it is conceivable for such a configuration to be unstable. Indeed, in the absence of viscosity, non-linear stability is not known to be true for most fluid equations, except in very specific cases (see, for example, \cite{VillaniMouhot}, \cite{BedrossianMasmoudi}, \cite{Elgindi}).\

In this work, we prove that a particular stationary configuration is \emph{non-linearly stable} on an interval of time which is proportional to the $4/3$-power of the inverse of the size of the perturbation. Prior to this work, the long-time existence of non-trivial solutions was unknown. We emphasize that in the viscous case, stability of the stationary solution we will study is trivial. However, in the invsicid case, even \emph{linear} stability seems to have been unknown prior to this work.

\subsection{Plan of the Article}
We start by considering the dispersive properties of the linear evolution of our problems in section \ref{sec:dispersion}. Given the lack of isotropy in the equation one might expect a slower rate of decay than for other dispersive $2d$ models. Proposition \ref{prop:disp_est} shows that this is indeed the case -- more precisely, solutions to the linear equation decay at a rate of $t^{-1/2}$ in $L^\infty$. With the help of special functions we can then show that this decay is sharp.

In section \ref{sec:long_solutions} we demonstrate how to exploit the dispersion to obtain an increased time of existence of solutions of the SQG equation. This is the content of Theorem \ref{thm:long_solutions}.

Finally we show how the same semigroup as for the SQG equation comes up in the question of stability of solutions to the Boussinesq equation (section \ref{sec:bouss}) and directly obtain a corresponding stability result (Theorem \ref{thm:bouss}).

\subsubsection*{Notation}
On $\R^2$ we shall denote by $R_j$ the Riesz transform in the variable $x_j$, $1\leq j\leq 2$, given by
\begin{equation*}
 R_jf(x)=-\int_{\R^2}i\frac{\xi_j}{\abs{\xi}}\hat{f}(\xi)e^{ix\cdot\xi}\, d\xi.
\end{equation*}
For $s\in\R$ and $1\leq p\leq\infty$ we denote by $W^{s,p}$ the inhomogeneous Sobolev space with derivatives up to order $s$ in $L^p$, and write $H^s:=W^{s,2}$, while $\dot{H}^s$ stands for the homogeneous Sobolev space of derivatives of order $s$ in $L^2$. By $\dot{B}^{a}_{b,c}$ we shall denote the homogeneous Besov space of $a$-th derivatives of the frequency localized pieces in $L^b$, summed with respect to $\ell^c$.

\section{Dispersion}\label{sec:dispersion}
In this section we consider the dispersive properties of the linear equation
\begin{equation}\label{eq:lin_dsqg}
 \partial_t \omega-R_1 \omega=0,\qquad \omega:\R^+\!\!\times\R^2\rightarrow\R.
\end{equation}
Its semigroup is given by
\begin{equation*}
 e^{R_1 t}f(x)=\int_{\R^2}\hat{f}(\xi)e^{ix\cdot\xi -it\frac{\xi_1}{\abs{\xi}}}\, d\xi=\int_{\R^2}\hat{f}(\xi)e^{it\phi(\xi)}\,d\xi,
\end{equation*}
where we denote by $\phi$ the phase\footnote{We omit the variables $x,t$ from the notation since we regard them as fixed (but arbitrary).}
\begin{equation*}
 \phi(\xi):=\frac{x}{t}\cdot\xi-\frac{\xi_1}{\abs{\xi}}.
\end{equation*}

A direct computation gives the determinant of the Hessian (in $\xi$) of $\phi$ as $-\frac{\xi_2^2}{\abs{\xi}^6}.$ The degeneracy of this along $\{\xi_2=0\}$ allows for stationary points of the phase with a degenerate Hessian. Hence we only obtain a slower rate of decay of $t^{-1/2}$ (Proposition \ref{prop:disp_est}). As we we will show further below, this decay is actually sharp (Proposition \ref{prop:sharp_decay}).

\subsection{Dispersive Estimate}
\begin{prop}\label{prop:disp_est}
 Let $f\in C^\infty_c(\R^2)$. Then
 \begin{equation}\label{eq:disp_est}
  \norm{e^{R_1 t}f}_{L^\infty}\leq C\, t^{-1/2}\norm{f}_{\dot{B}^{2}_{1,1}}
 \end{equation}
 for some constant $C>0$.
\end{prop}

\begin{proof}
 We will localize $f$ in frequency (using Littlewood-Paley theory), estimate the different pieces separately and sum in the end.

\emph{1. Localization:} For $j\in\Z$ let $P_j$ be the Littlewood-Paley projectors associated to a smooth bump function $\varphi : \R^2\rightarrow\R$ with support near $\abs{\xi}\sim 1$, i.e.\ $\widehat{P_j f}(\xi)=\varphi(2^{-j}\xi)\hat{f}(\xi)$. 
 We note that by Young's convolution inequality
   \begin{equation*}
   \begin{aligned}
    \norm{e^{R_1 t}f}_{L^\infty}&=\norm{e^{R_1 t}(\sum_{j\in\Z} P_j f)}_{L^\infty}\leq \sum_{j\in\Z}\norm{e^{R_1 t}P_j f}_{L^\infty}\\
     &=\sum_{j\in\Z}\norm{\iFT{e^{-it\frac{\xi_1}{\abs{\xi}}}\varphi(2^{-j}\xi)\hat{f}(\xi))}}_{L^\infty}\\
     &=\sum_{j\in\Z}\norm{\iFT{e^{-it\frac{\xi_1}{\abs{\xi}}}\varphi(2^{-j}\xi)}\ast\iFT{\tilde{\varphi}(2^{-j}\xi)\hat{f}(\xi))}}_{L^\infty}\\
     &\leq \sum_{j\in\Z}\norm{\iFT{e^{-it\frac{\xi_1}{\abs{\xi}}}\varphi(2^{-j}\xi)}}_{L^\infty}\norm{\iFT{\tilde{\varphi}(2^{-j}\xi)\hat{f}(\xi))}}_{L^1}\\
     &=\sum_{j\in\Z}\norm{\iFT{e^{-it\frac{\xi_1}{\abs{\xi}}}\varphi(2^{-j}\xi)}}_{L^\infty}\norm{Q_j f}_{L^1},
   \end{aligned}
   \end{equation*}
 where $Q_j$ is a Littlewood-Paley projector associated to a slightly ``fattened'' bump function $\tilde{\varphi}$ (which equals $1$ on the support of $\varphi$).

 Hence it suffices to estimate the semigroup on the frequency localizers $\varphi$. To this end we notice that by a change of variables
 \begin{equation*}
  \begin{aligned}
   (e^{R_1 t}P_j)(x)&=\int_{\R^2}\varphi(2^{-j}\xi)e^{ix\cdot\xi -it\frac{\xi_1}{\abs{\xi}}}\, d\xi=2^{2j}\int_{\R^2}\varphi(\xi)e^{i(2^j x)\cdot\xi -it\frac{\xi_1}{\abs{\xi}}}\, d\xi\\
    &=2^{2j}(e^{R_1 t}P_0)(2^j x).
  \end{aligned}
 \end{equation*}
 Since we will estimate this in $L^\infty$ we can further reduce to estimating
  \begin{equation}\label{key_piece}
   \norm{e^{R_1 t}P_0}_{L^\infty}
  \end{equation}
 and our final estimate will be
 \begin{equation}\label{ineq:final_step1}
  \begin{aligned}
  \norm{e^{R_1 t}f}_{L^\infty}&\lesssim \sum_{j\in\Z} 2^{2j}\norm{e^{R_1 t}P_0}_{L^\infty}\norm{Q_j f}_{L^1}=\norm{e^{R_1 t}P_0}_{L^\infty}\sum_{j\in\Z} 2^{2j}\norm{Q_j f}_{L^1}\\
  &=\norm{e^{R_1 t}P_0}_{L^\infty}\norm{f}_{\dot{B}^{2}_{1,1}}.
  \end{aligned}
 \end{equation}

\emph{2. Estimate of \eqref{key_piece}:} 
 We want to estimate in $L^\infty$
  \begin{equation*}\label{eq:local_piece}
   e^{R_1 t}P_0(x)=\int_{\abs{\xi}\sim 1}\varphi(\xi)e^{ix\cdot\xi -it\frac{\xi_1}{\abs{\xi}}}\, d\xi = \int_{\abs{\xi}\sim 1}\varphi(\xi)e^{it\phi(\xi)}\, d\xi.  
  \end{equation*}
 To this end we recall that the determinant of the Hessian $H_\phi$ of $\phi$ is given by
 \begin{equation*}
  \textnormal{det}H_\phi(\xi)=-\frac{\xi_2^2}{\abs{\xi}^6}.
 \end{equation*}
 Next we split the domain of integration in the above integral in two regions: one is small and contains the degenerate stationary points, the other is large but free of degenerate stationary points. More precisely, for a parameter $\lambda>0$ to be chosen later we have
 \begin{equation*}
  e^{R_1 t}P_0(x)=\int_{\abs{\xi}\sim 1,\, \abs{\xi_2}\leq\lambda}\varphi(\xi)e^{it\phi(\xi)}\, d\xi+\int_{\abs{\xi}\sim 1,\, \abs{\xi_2}>\lambda}\varphi(\xi)e^{it\phi(\xi)}\, d\xi.
 \end{equation*}
 Since the integrand is a bounded function we can estimate the first term by
 \begin{equation*}\label{eq:piece1}
  \abs{\int_{\abs{\xi}\sim 1,\, \abs{\xi_2}\leq\lambda}\varphi(\xi)e^{it\phi(\xi)}\, d\xi} \lesssim \lambda.
 \end{equation*}
 For the second term we note that the phase is stationary when
 \begin{equation*}
  \nabla_\xi\phi=\frac{x}{t}+\abs{\xi}^{-2}\left(\xi_2^2,\xi_1\xi_2\right)=0,
 \end{equation*}
 so that for every $(t,x)\in\R\times\R^2$ we can have at most two stationary points. Hence by the stationary phase lemma we can bound the second term by
 \begin{equation*}\label{eq:piece2}
 \begin{aligned}
  \abs{\int_{\abs{\xi}\sim 1,\, \abs{\xi_2}>\lambda}\varphi(\xi)e^{it\phi(\xi)}\, d\xi}& \lesssim \sum_{\eta\textnormal{ stationary},\, \abs{\eta}\sim 1,\, \abs{\eta_2}>\lambda} t^{-1}\abs{\textnormal{det}H_\phi(\eta)}^{-\frac{1}{2}}\\
  &\hspace{-2cm}\lesssim \sum_{\eta\textnormal{ stationary},\, \abs{\eta}\sim 1,\, \abs{\eta_2}>\lambda} t^{-1}\frac{\abs{\eta}^3}{\abs{\eta_2}}\\
  &\hspace{-2cm}\lesssim t^{-1}\lambda^{-1}.
 \end{aligned}
 \end{equation*}
 Combining these two estimates yields the desired bound
 \begin{equation*}
  \norm{e^{R_1 t}P_0}_{L^\infty}\lesssim \lambda+t^{-1}\lambda^{-1}=t^{-\frac{1}{2}}
 \end{equation*}
 once we choose $\lambda=t^{-\frac{1}{2}}$. In conjunction with \eqref{ineq:final_step1} this concludes the proof of the proposition.
\end{proof}

\begin{rema}[Dispersion for models with $1<\alpha\leq 2$]\label{rema:alpha-disp}
The SQG equation can be viewed as the special case $\alpha=1$ of the broader class of models
\begin{equation*}
  \partial_t\theta+u\cdot\nabla\theta=\frac{\partial_1}{\abs{\nabla}^\alpha}\theta,\qquad 1\leq\alpha\leq 2,
\end{equation*}
with $u=\nabla^\perp(-\Delta)^{-\alpha/2}\theta=(-\frac{\partial_2}{\abs{\nabla}^\alpha} \theta,\frac{\partial_1}{\abs{\nabla}^\alpha} \theta).$ We note that for $\alpha=2$ this yields the well-known \emph{$\beta$-plane equation} (see e.g.\ \cite{AIP_1.3141499}).

It turns out that for the cases $\alpha>1$ the dispersion of the linear semigroups can be computed using standard stationary phase methods. More precisely, we have:
\begin{equation*}
 e^{\frac{\partial_1}{\abs{\nabla}^\alpha} t}f(x)=\int_{\R^2}\hat{f}(\xi)e^{ix\cdot\xi -it\frac{\xi_1}{\abs{\xi}^\alpha}}\, d\xi=\int_{\R^2}\hat{f}(\xi)e^{it\phi_\alpha(\xi)}\,d\xi,
\end{equation*}
where
\begin{equation*}
\phi_\alpha(\xi):=\frac{x}{t}\xi-\frac{\xi_1}{\abs{\xi}^\alpha}.
\end{equation*}
Then
\begin{equation*}
 \nabla_\xi\phi_\alpha=\frac{x}{t}+\abs{\xi}^{-\alpha-2}\left((\alpha-1)\xi_1^2+\xi_2^2,\alpha\xi_1\xi_2\right),
\end{equation*}
so that for every $(t,x)\in\R\times\R^2$ there are at most two stationary points. A direct computation then gives the determinant of the Hessian of $\phi_\alpha$ as
\begin{equation*}
 \det H_{\phi_\alpha}=\alpha^2\frac{(\alpha-1)\xi_1^2+\xi_2^2}{\abs{\xi}^{4+2\alpha}},
\end{equation*}
which is non-degenerate away from the origin.

Rescaling as in the proof of Proposition \ref{prop:disp_est} now yields
\begin{equation*}
\begin{aligned}
 e^{\frac{\partial_1}{\abs{\nabla}^\alpha} t}P_j(x)&=\int_{\R^2}\varphi(2^{-j}\xi)e^{i x\cdot\xi -it\frac{\xi_1}{\abs{\xi}^\alpha}}\,d\xi\\
 &=2^{2j}\int_{\R^2}\varphi(\xi)e^{i2^jx\cdot\xi -it 2^{j(1-\alpha)}\frac{\xi_1}{\abs{\xi}^\alpha}}\,d\xi\\
 &=2^{2j}e^{\frac{\partial_1}{\abs{\nabla}^\alpha} 2^{j(1-\alpha)} t}P_0(2^j x),
\end{aligned}
\end{equation*}
so that we obtain
\begin{equation*}
 \norm{e^{\frac{\partial_1}{\abs{\nabla}^\alpha} t}P_j}_{L^\infty}=2^{2j}\norm{e^{\frac{\partial_1}{\abs{\nabla}^\alpha} 2^{j(1-\alpha)} t}P_0}_{L^\infty}\lesssim 2^{2j} (2^{j(1-\alpha)} t)^{-1}=2^{j(1+\alpha)}t^{-1}
\end{equation*}
from the stationary phase lemma.
The full dispersive estimate thus reads
\begin{equation}\label{eq:alpha-disp-est}
 \norm{e^{\frac{\partial_1}{\abs{\nabla}^\alpha} t}f}_{L^\infty}\lesssim t^{-1}\norm{f}_{\dot{B}^{1+\alpha}_{1,1}},\quad 1<\alpha\leq 2.
\end{equation}
\end{rema}

As announced we now turn to the question of the sharpness of the dispersive estimate.

\subsection{Sharpness of the Dispersive Estimate}
It turns out that for all smooth, radial $L^1$ functions $f$ with $f(0)\neq 0$, the linear evolution $e^{Rt} f$ decays no faster than $\frac{1}{\sqrt{t}}$ in $L^\infty$:

\begin{prop}\label{prop:sharp_decay}
Let $f$ be a smooth, radial $L^1$ function with $f(0)\neq 0$. 

Then
$$e^{Rt}f(0)\sim f(0)\sqrt{\frac{2}{\pi t}}\cos(t-\frac{\pi}{4}) \quad \text{as} \quad t\rightarrow \infty.$$

\end{prop}
\begin{proof}
Let $f$ be a smooth, radial function. Then $\hat{f}$ is radial, so in particular

\begin{equation*}
 \begin{aligned}
  e^{Rt}f(0)&=\int_{\mathbb{R}^2} e^{-i\frac{\xi_1}{|\xi|}t}\hat{f}(\xi)\,d\xi=\int_{0}^\infty \hat{f}(r)r\,dr\int_{0}^{2\pi}e^{-it\cos(\vartheta)}\,d\vartheta\\
	    &=f(0)\int_{0}^{2\pi}e^{-it\cos(\vartheta)}d\vartheta.
 \end{aligned}
\end{equation*}

It turns out that $$\int_{0}^{2\pi}e^{-it\cos(\vartheta)}\,d\vartheta=J_0(t),$$ where $J_0$ is the zeroth order Bessel function of the first kind. The asymptotics of this function are well known:

\begin{lemm}[see page 364 in \cite{MR0167642}]
For $t\gg 1$ $$J_0(t)\sim \sqrt{\frac{2}{\pi t}}\cos(t-\frac{\pi}{4}).$$
\end{lemm}

This concludes the proof of the proposition. 
\end{proof}

\section{Increased Time of Existence of Solutions to the SQG Equation}\label{sec:long_solutions}
Consider the equation
\begin{equation}\label{eq:DSQG}
 \partial_t\theta+u\cdot\nabla\theta=R_1\theta
\end{equation}
for $\theta:\R^+\!\!\times\R^2\rightarrow\R$ and
$$u=\nabla^\perp(-\Delta)^{-1/2}\theta=(-R_2 \theta,R_1 \theta),$$
where $R_j$ denotes the Riesz transform in $x_j$, $j=1,2$.

Using the energy method and the Gagliardo-Nirenberg inequality one can prove the following blow-up criterion:
\begin{lemm}\label{lem:energy_ineq}
 Let $\theta$ be a solution of the dispersive SQG equation \eqref{eq:DSQG} defined on a time interval containing $[0,T]$. Then for any $s>0$ we have the bound
 \begin{equation}\label{eq:energy_ineq}
  \norm{\theta(T)}_{H^s}\leq \norm{\theta_0}_{H^s} \exp\left(c\int_{0}^T\norm{\nabla u(t)}_{L^\infty}+\norm{\nabla\theta(t)}_{L^\infty}dt\right)
 \end{equation}
 for some universal constant $c>0$.
\end{lemm}

\begin{proof}
We note that for the operator $J^s:=(1-\Delta)^{s/2}$ we have $\norm{J^s\theta}_{L^2}\sim\norm{\theta}_{H^s}$, so it suffices to bound $J^s\theta$ in $L^2$. Applying $J^s$ to the equation we see that
$$\partial_tJ^s\theta+ J^s(u\cdot\nabla\theta)=J^s R_1\theta=R_1 J^s\theta.$$
Next we multiply by $J^s\theta$ and integrate over $\R^2$ to obtain
\begin{equation}\label{eq:JsSQG2}
 \partial_t \norm{J^s\theta}_{L^2}^2+\ip{J^s(u\cdot\nabla\theta),J^s\theta}_{L^2}=0,
\end{equation}
since for any $f\in L^2$ we have $\int f \,R_1 f=0$.

We note that $u$ is divergence-free, so
$$\ip{u\cdot\nabla (J^s\theta),J^s\theta}_{L^2}=0,$$
and thus by the Gagliardo-Nirenberg inequality
\begin{equation*}
\abs{\ip{J^s(u\cdot\nabla\theta),J^s\theta}_{L^2}}\lesssim \norm{J^s\theta}_{L^2}^2(\norm{\nabla u}_{L^\infty}+\norm{\nabla\theta}_{L^\infty}).
\end{equation*}
Combining this with \eqref{eq:JsSQG2} yields
$$\partial_t \norm{J^s\theta}_{L^2}^2\lesssim (\norm{\nabla u}_{L^\infty}+\norm{\nabla\theta}_{L^\infty})\norm{J^s\theta}_{L^2}^2,$$
so we need only appeal to Gr\"onwall's Lemma to finish the proof.
\end{proof}

We are now in the position to prove one of the main theorems of this article, thus demonstrating that solutions to the dispersive surface quasi-geostrophic equation \eqref{eq:DSQG} exist for a longer period of time than what standard hyperbolic estimates predict:
\begin{thm}\label{thm:long_solutions}
 Let $\delta>0$, $\mu>0$. If $\norm{\theta_0}_{H^{4+\delta}}, \norm{\theta_0}_{W^{3+\mu,1}}\leq\epsilon$ there exists a unique solution  $\theta\in C^0([0,T];H^{4+\delta})$ of the initial value problem
 \begin{equation}
  \begin{cases}
   &\partial_t\theta+u\cdot\nabla\theta=R_1\theta,\\
   &u=\nabla^\perp(-\Delta)^{-1/2}\theta,\\
   &\theta(0)=\theta_0\in H^{4+\delta},
  \end{cases}
 \end{equation}
 with $T\sim \epsilon^{-4/3}$.

 Moreover, for $t\in[0,T]$ we have $\norm{\theta(t)}_{H^{4+\delta}}\lesssim\epsilon$.
\end{thm}

The proof proceeds by bounding $\norm{\nabla u}_{L^\infty}$ and $\norm{\nabla \theta}_{L^\infty}$ using the dispersion and then invoking our previous Lemma \ref{lem:energy_ineq}.

\begin{proof}
By Duhamel's Principle we have
\begin{equation*}
\theta(t)=e^{R_1 t}\theta_0+\int_{0}^te^{R_1(t-s)}\left(u(s)\cdot\nabla\theta(s)\right)ds.
\end{equation*}

For notational simplicity we introduce the two first-order pseudo-differential operators $L_1:=\nabla\nabla^\perp(-\Delta)^{-1/2}$ and $L_2:=\nabla$, for which $L_1 \theta=\nabla u$ and $L_2\theta=\nabla \theta$. 

Upon applying $L_i$ to \eqref{eq:DSQG} and noting that $e^{R_1\cdot}$ commutes with $L_i$ we get
$$L_i\theta(t)= e^{R_1t}L_i\theta_0+\int_{0}^te^{R_1(t-s)}L_i(u\cdot\nabla \theta)(s)\,ds.$$

Thus for any $\delta, \mu>0$ the decay estimate \eqref{eq:disp_est} implies
\begin{equation*}
 \begin{aligned}
  \norm{L_i\theta(t)}_{L^\infty}&\lesssim \frac{1}{\sqrt{t+1}}\norm{L_i\theta_0}_{\dot{B}^2_{1,1}}+\int_{0}^{t}\frac{1}{\sqrt{t-s+1}}\norm{L_i(u\cdot\nabla\theta)}_{\dot{B}^2_{1,1}}\,ds\\
  &\lesssim \frac{1}{\sqrt{t+1}}\norm{\theta_0}_{\dot{B}^3_{1,1}}+\int_{0}^{t}\frac{1}{\sqrt{t-s+1}}\norm{u\cdot\nabla\theta}_{\dot{B}^3_{1,1}}\,ds\\
  &\lesssim \frac{1}{\sqrt{t+1}}\norm{\theta_0}_{W^{3+\mu,1}}+\int_{0}^{t}\frac{1}{\sqrt{t-s+1}}\norm{u\cdot\nabla\theta}_{W^{3+\delta,1}}\,ds\\
  &\lesssim \frac{1}{\sqrt{t+1}}\norm{\theta_0}_{W^{3+\mu,1}}+\int_{0}^{t}\frac{1}{\sqrt{t-s+1}}\norm{\theta}_{H^{4+\delta}}^2\,ds.
 \end{aligned}
\end{equation*}

Now suppose that $\norm{\theta}_{H^{4+\delta}}\leq 2\epsilon$ on a time interval $[0,T].$ \footnote{We know that the maximal time interval $[0,T_{max}]$ where this is true has $T_{max}\geq c/\epsilon.$} Then for $t\leq T$
$$\norm{L_i \theta(t)}_{L^\infty}\leq c\epsilon\frac{1}{t^{1/2}}+\int_{0}^{t}\frac{c}{(t-s+1)^{1/2}}\epsilon^2\,ds\leq c\Big( \frac{\epsilon}{t^{1/2}}+\epsilon^2 t^{1/2}\Big)$$
for a universal constant $c>0$.

Now, using the energy inequality \eqref{eq:energy_ineq} from Lemma \ref{lem:energy_ineq} we obtain
$$\norm{\theta(t)}_{H^s} \leq \norm{\theta_{0}}_{H^s}e^{c\int_{0}^{T}\frac{\epsilon}{t^{1/2}}+\epsilon^2 t^{1/2}\;dt}= \norm{\theta_0}_{H^s}e^{c\left(\epsilon T^{1/2}+\epsilon^2 T^{3/2}\right)}.$$

Choosing $s=4+\delta$ we see that $\norm{\theta}_{H^{4+\delta}}\leq 2\epsilon$ so long as $$\epsilon T^{1/2}+\epsilon^2 T^{3/2}\lesssim 1,$$ i.e.\ so long as $T\lesssim \epsilon^{-4/3}.$ This completes the proof of the theorem. 

\end{proof}

\begin{rema}[Longer Time of Existence for models with $1<\alpha\leq 2$]
 For the related models discussed in Remark \ref{rema:alpha-disp} (with parameter $1<\alpha\leq 2$) one may go through the analogous steps of the above proof to obtain the condition $$T\left(\log(T)-1\right)\lesssim\epsilon^{-2},$$ thanks to the higher rate of dispersion \eqref{eq:alpha-disp-est}. Hence we obtain an existence time of $T\sim \frac{\epsilon^{-2}}{\abs{\log(\epsilon)}}$ for solutions in $H^{3+\alpha+\delta}$ with small initial data in $W^{2+\alpha+\mu,1}$ and $H^{3+\alpha+\delta}$ (for any given $\mu,\delta>0$). 
\end{rema}

\section{A Stability Result for the Inviscid Boussinesq System}\label{sec:bouss}

Recall from the introduction (page \pageref{ssec:bouss_intro}) the inviscid Boussinesq system
\begin{equation}\label{eq:bouss}
\begin{cases}
 &\partial_t \omega+u\cdot\nabla\omega=\partial_x\rho,\\
 &\partial_{t}\rho+u\cdot\nabla\rho=0,\\
 &u=\nabla^\perp(-\Delta)^{-1}\omega,
\end{cases}
\end{equation}
where $\omega,\rho:\R^+\!\!\times\R^2\rightarrow\R$.

In this section we draw the connection to the linear decay estimate of section \ref{sec:dispersion} by linearizing the equations. We then show how this can provide a stability result for a certain stationary solution to the Boussinesq system \eqref{eq:bouss}.

To begin we notice that $\rho(t,x,y)=f(y)$ and $u\equiv\omega\equiv 0$ is a stationary solution of the Boussinesq system for any continuously differentiable function $f\in C^1$. Here we want to study the stability of the special stratified solution with $f(y)=-y.$

\subsection{Linearized Equations}
The linearized equations around $(\rho,u,\omega)=(-y,0,0)$ read
\begin{equation}\label{eq:lin_bouss}
 \begin{cases}
  &\partial_{t}\omega-\partial_{x}\rho=0,\\
  &\partial_{t}\rho-u_{2}=0,\\
  &u_{2}=\partial_{x}(-\Delta)^{-1}\omega.
 \end{cases}
\end{equation}
Thus
$$\partial_{tt}\omega=\partial_{t}\partial_x\rho=\partial_x u_2=\partial_{xx}(-\Delta)^{-1}\omega.$$
Upon taking the Fourier transform of both sides of this equation we get
$$\partial_{tt}\hat\omega(\xi,t)=-\frac{\xi_{1}^2}{\abs{\xi}^2}\hat\omega(\xi,t),$$
the solutions of which are of the form
$$\hat\omega(\xi,t)=A(\xi)e^{-i\frac{\xi_1}{\abs{\xi}}t}+B(\xi)e^{i\frac{\xi_1}{\abs{\xi}}t},$$
from which one computes\footnote{Equivalently, one may note that the system \eqref{eq:lin_bouss} written in the variables $(\omega,\abs{\nabla}\rho)$ is diagonalized by $\omega+\abs{\nabla}\rho$ and $\omega-\abs{\nabla}\rho$ with eigenvalues $\pm i\frac{\xi_1}{\abs{\xi}}$.} $u=\nabla^\perp(-\Delta)^{-1}\omega$ and $\hat{\rho}(\xi,t)=-\abs{\xi}^{-1}A(\xi)e^{-i\frac{\xi_1}{\abs{\xi}}t}+\abs{\xi}^{-1}B(\xi)e^{i\frac{\xi_1}{\abs{\xi}}t}$. 

By Proposition \ref{prop:disp_est} this implies that solutions of the linearized problem near $(\rho,u,\omega)=(-y,0,0)$ decay at the rate of $t^{-\frac{1}{2}}$ in $L^\infty$.

\begin{rema}
 We note that the linearization around $(\rho,u,\omega)=(y,0,0)$ gives $\partial_{tt}\omega=-\partial_{xx}(-\Delta)^{-1}\omega$, so that the stationary solution $\rho=y$, $u\equiv\omega\equiv 0$ is \emph{linearly unstable}. In the presence of viscosity this instability has been studied before (see for example \cite[Chapter 8]{MR1925398}).
\end{rema}

\subsection{Nonlinear Stability}
Next we will show how this decay of the linear solutions can be used to establish the stability of the stationary solution $\rho=-y$, $u\equiv\omega\equiv 0$.

When perturbing around $(\rho,u,\omega)=(-y,0,0)$ we get the following system:
\begin{equation}\label{eq:bouss_pert}
 \begin{cases}
  &\partial_t \omega+u\cdot\nabla\omega=\partial_x\rho,\\
  &\partial_{t}\rho+u\cdot\nabla\rho=u_{2},\\
  &u=\nabla^\perp(-\Delta)^{-1}\omega.
 \end{cases}
\end{equation}

We may write this as a perturbation of the linear system \eqref{eq:lin_bouss}, namely
\begin{equation*}
 \begin{cases}
  &\partial_{t}\omega-\partial_{x}\rho=F(u,\omega),\\
  &\partial_{t}\rho-u_2=G(u,\rho),\\
  &u=\nabla^{\perp}(-\Delta)^{-1}\omega,\\
 \end{cases}
\end{equation*}
where $F=-u\cdot\nabla\omega$ and $G=-u\cdot\nabla\rho$ are quadratic.

We observe that these equations are of the same form as those in section \ref{sec:long_solutions}, so the theory developed there can be carried over directly and yields the following
\begin{thm}[Increased time of stability of $(\rho,u,\omega)=(-y,0,0)$]\label{thm:bouss}
 Let $\delta>0$, $\mu>0$, $\gamma>0$. If $\norm{\omega_0}_{H^{4+\delta}}, \norm{\omega_0}_{H^{-1}}, \norm{\omega_0}_{W^{3+\mu,1}},\norm{\rho_0}_{H^{5+\gamma}}, \norm{\rho_0}_{W^{4+\mu,1}}\leq\epsilon$, then there exist unique functions  $\rho\in C^0([0,T];H^{5+\gamma})$, $\omega\in C^0([0,T];H^{4+\delta}\cap H^{-1})$ such that the initial value problem for the Boussinesq system \eqref{eq:bouss},
 \begin{equation*}
  \begin{cases}
   &\partial_t \omega+u\cdot\nabla\omega=\partial_x\rho,\\
   &\partial_{t}\rho+u\cdot\nabla\rho=0,\\
   &u=\nabla^\perp(-\Delta)^{-1}\omega,\\
   &\omega(0)=\omega_0\in H^{4+\delta}\cap H^{-1},\\
   &\rho(0)=\rho_0\in H^{5+\gamma},
  \end{cases}
 \end{equation*}
 is solved by $(\rho -y, u,\omega)$. Here, $T\sim \epsilon^{-4/3}$.

 In particular, small perturbations in $H^{4+\delta}\cap H^{-1}$ and $H^{5+\gamma}$ for $\omega$ and $\rho$ (respectively) of the special solution $(\rho,u,\omega)=(-y,0,0)$ of \eqref{eq:bouss} remain small for times $\lesssim\epsilon^{-4/3}$.
\end{thm}

\begin{rema}
 Given that $\omega$ is the vorticity of the fluid under consideration the requirement that $\norm{\omega_0}_{H^{-1}}\leq \epsilon$ is physically sensible: It corresponds to the fluid having small kinetic energy.
\end{rema}

\begin{proof}
 Essentially the proof is the same as that of Theorem \ref{thm:long_solutions}, so we only go over the two key ingredients: The dispersive estimate and the blow-up criterion.

  \emph{Dispersive Estimate.} As observed above, writing the equations for $(\rho -y, u,\omega)$ gives the system \eqref{eq:bouss_pert}, the linear part \eqref{eq:lin_bouss} of which is
 \begin{equation}\label{eq:lin_bouss2}
  \begin{cases}
   &\partial_{t}\omega-\partial_{x}\rho=0,\\
   &\partial_{t}\rho-\partial_{x}(-\Delta)^{-1}\omega=0,
  \end{cases}
 \end{equation}
 and shares the decay properties of the SQG equation, hence the analogous estimates
 \begin{equation*}
  \norm{(\abs{\nabla}\rho,\omega)(t)}_{L^\infty}\lesssim t^{-1/2}\norm{(\abs{\nabla}\rho_0,\omega_0)}_{\dot{B}^{2}_{1,1}}.
 \end{equation*}

 \emph{Blow-Up Criterion.} We proceed in analogy with the SQG case: For $s\geq -1$ we take $s$ derivatives of the first, and $s+1$ derivatives of the second equation in \eqref{eq:lin_bouss2}, multiply by $\abs{\nabla}^s\omega$ and $\abs{\nabla}^{s+1}\rho$ respectively, and integrate over $\R^2$ to obtain
  \begin{equation*}
   \begin{aligned}
    \frac{1}{2}\partial_t\norm{\abs{\nabla}^s\omega}_{L^2}&=\ip{\abs{\nabla}^s\partial_x\rho,\abs{\nabla}^s\omega}_{L^2}-\ip{\abs{\nabla}^s(u\cdot\nabla\omega),\abs{\nabla}^s\omega}_{L^2},\\
    \frac{1}{2}\partial_t\norm{\abs{\nabla}^{s+1}\rho}_{L^2}&=\ip{\abs{\nabla}^{s+1}\partial_{x}(-\Delta)^{-1}\omega,\abs{\nabla}^{s+1}\rho}_{L^2}-\ip{\abs{\nabla}^{s+1}(u\cdot\nabla\rho),\abs{\nabla}^{s+1}\rho}_{L^2}.
   \end{aligned}
  \end{equation*}
 We add these two equalities to get
 \begin{equation*}
   \frac{1}{2} \partial_t\left(\norm{\omega}_{\dot{H}^s}+\norm{\rho}_{\dot{H}^{s+1}}\right)=-\ip{\abs{\nabla}^s(u\cdot\nabla\omega),\abs{\nabla}^s\omega}_{L^2}-\ip{\abs{\nabla}^{s+1}(u\cdot\nabla\rho),\abs{\nabla}^{s+1}\rho}_{L^2},
 \end{equation*}
 since by integration by parts
\begin{equation*}
 \ip{\abs{\nabla}^s\partial_x\rho,\abs{\nabla}^s\omega}_{L^2}+\ip{\abs{\nabla}^{s+1}\partial_{x}(-\Delta)^{-1}\omega,\abs{\nabla}^{s+1}\rho}_{L^2}=0.
\end{equation*}
But from this it follows by Gagliardo-Nirenberg estimates and Gr\"onwall's Lemma (as in the proof of Lemma \ref{lem:energy_ineq}) that
\begin{equation*}
\begin{aligned}
 \norm{\omega(T)}_{\dot{H}^s}+\norm{\rho(T)}_{\dot{H}^{s+1}}&\lesssim \left(\norm{\omega_0}_{\dot{H}^s}+\norm{\rho_0}_{\dot{H}^{s+1}}\right)\times\\
  &\quad\exp\left(c\int_0^T\norm{\nabla u(t)}_{L^\infty}+\norm{\nabla\rho(t)}_{L^\infty}\;dt\right).
\end{aligned}
\end{equation*}
Summation now gives the bound in inhomogeneous Sobolev spaces.
\end{proof}

\begin{acknowledgements}
 The authors would like to thank Pierre Germain for his helpful comments.
\end{acknowledgements}

\bibliographystyle{plain}
\bibliography{sqgrefs.bib}

\end{document}